%% file: main.tex
\documentclass[10pt]{article}
\input{_macros}

\definecolor{OliveGreen}{rgb}{0,0.6,0}
\usepackage{custom}
\usepackage[normalem]{ulem}

\allowdisplaybreaks{}
\makeatletter
\renewcommand{\paragraph}{%
  \@startsection{paragraph}{4}%
  {\z@}{1.25ex \@plus 1ex \@minus .2ex}{-1em}%
  {\normalfont\normalsize\bfseries}%
}
\makeatother

\begin{document}

\title{Uniform-in-$N$ log-Sobolev inequality for the \\ mean-field Langevin dynamics with convex energy}
\author{\!\!\!\!\!
 Sinho Chewi\thanks{
  Department of Statistics and Data Science at
  Yale University, \texttt{sinho.chewi@yale.edu}
 }
 \ \ \ \ \
 Atsushi Nitanda\thanks{
 Center for Frontier AI Research at
 Agency for Science, Technology, and Research (A$\star$STAR), and
 College of Computing and Data Science at Nanyang Technological University, \texttt{atsushi\_nitanda@cfar.a-star.edu.sg}
 }
 \ \ \ \ \
Matthew S.\ Zhang\thanks{
  Department of Computer Science at
  University of Toronto, and Vector Institute, \texttt{matthew.zhang@mail.utoronto.ca}
}
}

\maketitle

\begin{abstract}
    We establish a log-Sobolev inequality for the stationary distribution of mean-field Langevin dynamics with a constant that is independent of the number of particles $N$.
    Our proof proceeds by establishing the existence of a Lipschitz transport map from the standard Gaussian measure via the reverse heat flow of Kim and Milman.
\end{abstract}

\section{Introduction}

In this paper, we establish a log-Sobolev inequality for the stationary measure of the interacting particle approximation of the mean-field Langevin dynamics, with a constant independent of the number of particles $N$.
The key assumption we adopt is that the mean-field energy functional is convex along linear interpolations of the measure, although our proof requires additional smoothness and structural conditions to conclude.
We first provide some background before fully describing the result.

Let $\mc P_2(\R^d)$ denote the space of probability measures over $\R^d$ with finite second moment, and let $\mc F : \mc P_2(\R^d)\to\R$ be a functional over this space.
The associated mean-field Langevin dynamics is the McKean--Vlasov equation
\begin{align}\label{eq:mean-field-langevin}
    \D X_t = -\nabla_{\mc W_2} \mc F(\xi_t, X_t) \, \D t + \sigma \, \D B_t\,, \qquad \xi_t \deq \law(X_t)\,,
\end{align}
where ${(B_t)}_{t\ge 0}$ is a standard Brownian motion on $\R^d$ and $\nabla_{\mc W_2} \mc F(\mu,\cdot) = \nabla \delta \mc F(\mu)$, the gradient of the first variation, can be interpreted as the Wasserstein gradient of $\mc F$ (see~\cite{AGS}).
In turn, the curve of measures ${(\xi_t)}_{t\ge 0}$ admits an interpretation as the Wasserstein gradient flow of the entropically regularized mean-field energy functional $\mc E: \mc P_2(\R^d) \to \R$ given by
\begin{align}\label{eq:mean-field-energy}
    \mc E(\nu) \deq \mc F(\nu) + \frac{\sigma^2}{2} \int \log \nu \, \D \nu\,, \qquad\text{if}~\nu \ll \text{Leb}\,,
\end{align}
and $\mc E(\nu) = +\infty$ otherwise.
The first-order optimality condition for minimization of $\mc E$ suggests that any minimizer $\pi$ should satisfy the following (implicit) equation:
\begin{align}\label{eq:mean-field-measure}
    \pi(\D x) \propto \exp\Bigl(-\frac{2}{\sigma^2}\, \delta \mc F(\pi, x)\Bigr)\,\D x\,.
\end{align}

In this work, we are interested in the particle approximation of the mean-field system, described as follows.
Given a point $x^{1:N} \in \R^{N \times d}$, define the corresponding empirical measure to be
\begin{align*}
    \rho_{x^{1:N}} \deq \frac{1}{N} \sum_{i=1}^N \delta_{x^i}\,,
\end{align*}
where $\delta_{x^i}$ denotes the Dirac delta measure centered at $x^i$. Here $x^{1:N} = [x^1, \ldots, x^N]$, where $x^i \in \R^d$ for all $i \in [N]$.
Consider the following probability measure over $\R^{N\times d}$:
\begin{align}\label{eq:finite-particle-measure}
    \mu^{1:N}(\D x^{1:N}) \propto \exp \Bigl(-\frac{2N}{\sigma^2}\, \mc F(\rho_{x^{1:N}})\Bigr)\, \D x^{1:N}\,.
\end{align}
This is called the finite-particle approximation to~\eqref{eq:mean-field-measure}, and it describes the stationary measure of an interacting particle approximation to~\eqref{eq:mean-field-langevin}.
We are broadly interested in the approximation error incurred by the finite-particle approximation---which is generally known in the literature as \emph{propagation of chaos}---as well as the long-time convergence of the finite-particle system, as these two aspects govern the computational effort required to sample from the measure $\pi$ given by~\eqref{eq:mean-field-measure}~\cite{kook2024sampling}.

Without further assumptions, the problem is not well-posed, since a minimizer of the mean-field energy $\mc E$ may not be unique.
To recover uniqueness, one seeks conditions under which $\mc E$ is strictly convex in a suitable sense, and in the literature there are at least two distinct conditions which cover applications of rather different natures: one can assume that $\mc F$ is strictly convex along Wasserstein geodesics, which is also known as \emph{displacement convexity}, or that $\mc F$ is strictly convex along linear interpolations of the measure, which we abbreviate as \emph{linear convexity}. The latter condition has recently attracted attention from the theoretical machine learning community because it captures applications to two-layer neural networks in the mean-field regime \cite{nitanda2017stochastic, chizat2018global, mei2018mean, Chi22MFLangevin, rotskoff2022trainability}.

Recently, under the condition of linear convexity,~\cite{chen2022uniform, Nit24MeanField} obtained propagation of chaos bounds for the finite-particle approximation of the mean-field Langevin dynamics.
Our goal here is to study the second question listed above---namely, the long-time convergence of the finite-particle system---which amounts to a study of the log-Sobolev constant of $\mu^{1:N}$.
Recall that a measure $\nu \in \mc P(\R^d)$ satisfies a logarithmic Sobolev inequality with constant $C_{\LSI}(\nu)$ if for all smooth test functions $g: \R^d \to \R$, we have
\begin{align}\label{eq:LSI}\tag{LSI}
    \ent_\nu g^2 \leq 2C_{\LSI}(\nu) \E_\nu[\norm{\nabla g}^2]\,,
\end{align}
where $\ent_\nu f \deq \E_\nu f \log f - \E_\nu f \log \E_\nu f$ for a non-negative function $f$. \eqref{eq:LSI} has numerous ramifications, implying for instance sub-Gaussian concentration of Lipschitz functions around their mean under $\nu$ via the Herbst argument.
Moreover, it is well-known that~\eqref{eq:LSI} is equivalent to the exponential decay of the relative entropy along the Langevin dynamics toward $\nu$.
Consequently, obtaining tight bounds for the log-Sobolev constant of a measure is a ubiquitous problem in high-dimensional probability.

We further assume that $\mc F$ has the structure $\mc F(\nu) = \frac{\lambda}{2}\int\norm\cdot^2\,\D \nu + \mc F_0(\nu)$.
We remark that under suitable assumptions on $\mc F_0$, it is straightforward to establish a log-Sobolev inequality for $\mu^{1:N}$ through standard perturbation principles.
However, a na\"{\i}ve application of such arguments leads to a log-Sobolev constant which tends to infinity with the number of particles $N$.
The main result of this paper is to establish the (considerably trickier) result that~\eqref{eq:LSI} holds with a constant that is \emph{independent} of $N$.

\paragraph*{Prior work.}
To the best of our knowledge, the first result on a uniform-in-$N$ LSI under linear convexity (beyond the non-interacting case when $\mc F$ is linear in the measure) was established recently in~\cite{kook2024sampling}.
In that work, the authors showed the existence of a Lipschitz transport map from the standard Gaussian measure $\gamma$ to $\mu^{1:N}$, based on the reverse heat flow map of Kim and Milman~\cite{KimMil12ReverseHeat}. The proof combined recent heat flow estimates from~\cite{brigati2024heat} with a new propagation of chaos result for interacting particles in heterogeneous environments. However, due to a suboptimal estimation of the Lipschitz constant, the bound on the log-Sobolev constant therein scales doubly exponentially, i.e., $\exp(\exp(\Omega(\delta)))$, where $\delta$ is a measure of the size of the perturbation $\mc F_0$.
Needless to say, for even mild perturbations,
the doubly exponential bound implies vacuous bounds for any real-life applications.

We describe the source of the doubly exponential dependence, since it provides context for our result.
The Lipschitz constant of the reverse heat flow map depends exponentially on certain heat flow estimates.
In~\cite{kook2024sampling}, these heat flow estimates were obtained through the help of propagation of chaos bounds, which were based on the argument of~\cite{chen2022uniform}. However, the propagation of chaos bound of~\cite{chen2022uniform} incurs a dependence on the log-Sobolev constant of $\pi$ (and of certain ``proximal Gibbs measures'', see \cite{, nitanda2022convex}), which already scales exponentially in the perturbation strength.
Hence the double exponential.

\paragraph*{Our result and approach.}
We closely follow the approach of~\cite{kook2024sampling}, but we crucially improve the bound on the log-Sobolev constant to a single exponential: $\exp(\text{poly}(\delta))$.

The inspiration for our improvement is the recent propagation of chaos bound in~\cite{Nit24MeanField}, which improves upon~\cite{chen2022uniform} by removing the dependence on the log-Sobolev constant of the proximal Gibbs measures.
We generalize the result of~\cite{Nit24MeanField} to heterogeneous environments and combine it with the argument of~\cite{kook2024sampling} to obtain our improved bound on the LSI constant, stated precisely in \S\ref{sec:main_result}.
As we explain in \S\ref{sec:technical_overview}, this entails stitching together different heat flow estimates for short and long times.

\paragraph*{Concurrent work.}
While preparing this work, we became aware of the concurrent work of~\cite{Wang24UnifLSI}, which establishes a comparable result---namely, a uniform-in-$N$ LSI under linear convexity with a constant that scales (singly) exponentially in the size of the perturbation.
The remarkable proof of~\cite{Wang24UnifLSI} combines together a defective LSI from~\cite{chen2022uniform} with an ingenious uniform-in-$N$ Poincar\'e inequality based on the recent techniques of~\cite{Gui+22UnifLSI}.

The result of~\cite{Wang24UnifLSI} applies more generally than ours, although we note that our result is slightly stronger in that the existence of a Lipschitz transport from the Gaussian implies the validity of other functional inequalities. We give a detailed comparison in \S\ref{sec:main_result}.
In any case, we believe that our approach is still of interest due to the fundamentally different nature of the proof.

\paragraph{Notation.}
In this work, we use $\norm{\cdot}$ to refer to the $2$-norm on $\R^d$, and $\norm\cdot_{\op}$ for the operator norm on $\R^{d \times d}$. We let $\mc P_2(\R^d)$ denote the set of probability measures on $\R^d$ which have finite second moment. In general, we use superscripts to denote the indices of a particle; the notation $x^{1:N} \in \R^{N \times d}$ is to be interpreted as a vector $x^{1:N} = [x^1, \ldots, x^N]$, where each $x^i \in \R^d$ for $i \in [N]$.
We use $\mc O$ to denote asymptotic upper bounds up to universal constants, and $\Omega$ similarly to denote asymptotic lower bounds. If $f$ is $\mu$-integrable, then we use the notation $\inner{f, \mu}$ to denote $\E_\mu f$.

\section{Main result}\label{sec:main_result}

We assume that the energy $\mc F: \mc P_2(\R^d) \to \R$ has the perturbative form of $\mc F: \nu \mapsto \mc F_0(\nu) + \int V \, \D \nu$, where $V$ is $\lambda$-strongly convex.

In this work, we impose the following assumptions on $\mc F_0$, which are the same as the ones in~\cite{Wang24UnifLSI} except that we substitute boundedness of $\nabla^2 \delta \mc F_0$ with Assumption~\ref{ass:bddgrad}.

\begin{assumption}[Linear convexity of $\mc F_0$]\label{ass:cvxty}
    The functional $\mc F_0$ is linearly convex, in the sense that for any two $\nu_0, \nu_1 \in \mc P_2(\R^d)$ and $t \in [0, 1]$,
    \begin{align*}
        \mc F_0((1-t)\, \nu_0 + t\, \nu_1) \leq (1-t)\, \mc F_0(\nu_0) + t\, \mc F_0(\nu_1)\,.
    \end{align*}
\end{assumption}

\begin{assumption}[Smoothness]\label{ass:smoothness}
    Let $\delta^2 \mc F_0 : \mc P_2(\R^d)\times \R^d\times \R^d\to\R$ denote the second variation of $\mc F_0$.
    Then, the following bound holds.
    \begin{align*}
        \norm{\nabla_1\nabla_2 \delta^2 \mc F_0(\nu, x,y)}_{\rm op}
        &\le \betB\,, \qquad\forall\,\nu \in \mc P_2(\R^d)\,,\;x\in\R^d\,,\;y\in\R^d\,.
    \end{align*}
\end{assumption}

\begin{assumption}[Bounded gradient]\label{ass:bddgrad}
    The Wasserstein gradient of $\mc F_0$ is uniformly bounded by $B < \infty$, i.e.,
    \begin{align*}
        \norm{\nabla_{\mc W_2} \mc F_0(\nu, x)} \leq B\,, \qquad\forall\,\nu\in\mc P_2(\R^d)\,,\;x\in\R^d\,.
    \end{align*}
\end{assumption}

\begin{remark}\label{rmk:smoothness}
    Assumption~\ref{ass:smoothness} is also adopted in the prior works~\cite{chen2022uniform, kook2024sampling, Wang24UnifLSI}.
    On the other hand,~\cite{kook2024sampling} further assumes that for any $\nu,\nu' \in \mc P_2(\R^d)$ and any $x,x'\in\R^d$,
    \begin{align*}
        \norm{\nabla_{\mc W_2} \mc F_0(\nu, x) - \nabla_{\mc W_2} \mc F_0(\nu', x')}
        &\le \beta\,(\norm{x-x'} + \mc W_1(\nu,\nu'))\,.
    \end{align*}
    By taking $\nu = \nu'$, this implies boundedness of $\nabla^2 \delta \mc F_0(\nu)$, uniformly over $\nu \in \mc P_2(\R^d)$, which is also assumed in~\cite{chen2022uniform, Wang24UnifLSI}.
    In contrast, we avoid the latter assumption, which, as we explain below, allows our results to be applied to certain mean-field models of two-layer neural networks with ReLU activations.
\end{remark}

\begin{example}[Mean-field two-layer networks]\label{ex:nn}
    The primary example that we have in mind is when $\mc F_0$ is of the form $\mc F_0(\nu) = \int \ell(\int h(\theta, z)\,\nu(\D\theta), z)\, P(\D z)$, where $\ell : \R\times \mc Z\to\R$, $h : \R^d\times \mc Z\to\R$, and $P$ is an auxiliary probability measure over $\mc Z$.
    For example, this encompasses applications to two-layer neural networks in the mean-field regime, as we describe next.

    Let ${\{(X_i, Y_i)\}}_{i=1}^n \subseteq \mc X\times \mc Y$, $\mc Y \subseteq \R$, be a dataset and consider the problem of finding a predictor $f : \mc X\to\mc Y$ to minimize $\frac{1}{n}\sum_{i=1}^n \mc L(f(X_i), Y_i)$, where $\mc L : \mc Y\times \mc Y \to \R_+$ is a loss function.
    For example, for appropriate choices of $\mc L$, this covers both classification and regression tasks.
    We aim to learn the predictor $f$ within a parameterized family $\{f_\theta : \theta \in \R^d\}$, e.g., a family of neural networks, by minimizing the loss criterion $\theta \mapsto \frac{1}{n}\sum_{i=1}^n \mc L(f_\theta(X_i), Y_i)$.
    In the mean-field formulation, we replace the optimization over $\theta\in\R^d$ with an optimization over probability measures $\nu \in \mc P_2(\R^d)$ with loss $\nu \mapsto \frac{1}{n}\sum_{i=1}^n \mc L(\int f_\theta(X_i)\,\nu(\D \theta), Y_i) \eqqcolon \mc F_0(\nu)$.
    This is indeed of the form described above, with $P = \frac{1}{n} \sum_{i=1}^n \delta_{(X_i,Y_i)}$, $\ell(\hat y, z) = \mc L(\hat y, y)$, and $h(\theta, z) = f_\theta(x)$ for $z = (x,y)$.
    Our main results apply to a regularized version of this optimization problem where we instead seek to minimize $\nu \mapsto \mc F_0(\nu) + \int V\,\D \nu + \frac{\sigma^2}{2}\int \log \nu\,\D \nu$, where $V$ is strongly convex and can be interpreted as a regularization term for the parameters $\theta$, and the entropic term $\frac{\sigma^2}{2} \int \log \nu \, \D \nu$ introduces stochasticity which allows for global optimization guarantees.
    
    We assume that $\ell(\cdot, z)$ is convex, $L_\ell$-Lipschitz, and $\beta_\ell$-smooth, and that $h(\cdot,z)$ is $L_h$-Lipschitz, uniformly over $z\in \mc Z$.
    Then,
    \begin{align*}
        \delta \mc F_0(\nu, x)
        &= \int \ell'\bigl(\E_\nu h(\cdot,z),z\bigr)\,h(x,z)\,P(\D z)\,, \\
        \delta^2 \mc F_0(\nu, x, y)
        &= \int \ell''\bigl(\E_\nu h(\cdot,z), z\bigr)\,h(x,z)\,h(y,z)\,P(\D z)\,,
    \end{align*}
    and so the assumptions above hold with
    \begin{align*}
        \betB = L_h^2 \beta_\ell\,, \qquad B = L_h L_\ell\,.
    \end{align*}
    On the other hand, the boundedness of $\nabla^2 \delta \mc F_0$, as was assumed in~\cite{chen2022uniform, kook2024sampling, Wang24UnifLSI} requires smoothness of $h$ as well.

    In the example described above, and when $\mc X = \R^d$ and we take $h(\theta,z) = \ReLU(\langle \theta,x\rangle) \deq {(\langle \theta,x\rangle)}_+$, we can check that $h$ satisfies the Lipschitz assumption with $L_h = \max_{i\in [n]}{\norm{X_i}}$.
\end{example}

We now state our main result on the existence of an $L$-Lipschitz transport map from the standard Gaussian measure to $\mu^{1:N}$.

\begin{theorem}[Main result, general setting]\label{thm:main-generic}
   For any $N \geq 1$, under Assumptions~\ref{ass:cvxty},~\ref{ass:smoothness}, and~\ref{ass:bddgrad}, there exists an $L$-Lipschitz transport map from the standard Gaussian measure to $\mu^{1:N}$ with $L$ bounded by
   \begin{align*}
       L \lesssim \frac{\sigma}{\sqrt\lambda} \exp \Bigl(\mc{O}\Bigl(\frac{\beta d}{\lambda} +\frac{B^2}{\lambda\sigma^2} + \frac{\beta B^2 d}{\lambda^2\sigma^2} + \frac{\beta B^4}{\lambda^3\sigma^{4}} \Bigr)\Bigr)\,.
   \end{align*}
\end{theorem}

Moreover, if we assume that $\mc F_0$ takes on the form described in Example~\ref{ex:nn}, we can further refine the bound in Theorem~\ref{thm:main-generic}, removing a factor of the dimension from our constant.

\begin{theorem}[Main result, Example~\ref{ex:nn}]\label{thm:main-specific}
    For any $N \ge 1$, under Assumptions~\ref{ass:cvxty},~\ref{ass:smoothness}, and~\ref{ass:bddgrad} and in the setting of Example~\ref{ex:nn},
    there exists an $L$-Lipschitz transport map from the standard Gaussian measure to $\mu^{1:N}$ with $L$ bounded by
   \begin{align*}
       L \lesssim \frac{\sigma}{\sqrt\lambda} \exp \Bigl(\mc{O}\Bigl(\frac{L_h^2 \beta_\ell}{\lambda} + \frac{L_h^2 L_\ell^2}{\lambda \sigma^2} + \frac{L_h^6 L_\ell^4\beta_\ell}{\lambda^3\sigma^4} \Bigr)\Bigr)\,.
   \end{align*}
\end{theorem}

The existence of an $L$-Lipschitz transport map from the standard Gaussian measure to $\mu^{1:N}$ immediately implies, via~\cite[Proposition 5.4.3]{BGL14}, that $\mu^{1:N}$ satisfies~\eqref{eq:LSI} with $C_{\msf{LSI}} \le L^2$.
However, the Lipschitz transport map gives more, since it can be used to transfer a wide range of functional inequalities.
For example, by transferring~\cite[Corollary 8.5.4]{BGL14} from the Gaussian measure, we obtain that for every smooth function $f : \R^{d\times N}\to [0,1]$,
\begin{align*}
    \mc I\Bigl(\int f\,\D \mu^{1:N}\Bigr)
    &\le \int \sqrt{{\mc I(f)}^2 + L^2\,\norm{\nabla f}^2}\,\D \mu^{1:N}\,,
\end{align*}
where $\mc I : [0,1]\to \R$ is the Gaussian isoperimetric profile $\mc I \deq \phi \circ \Phi^{-1}$, with $\phi$, $\Phi$ denoting the PDF and CDF of the Gaussian density on $\R$ respectively.
In turn, it yields a Gaussian isoperimetric inequality: for every closed set $A \subseteq \R^{d\times N}$,
\begin{align}\label{eq:gaussian_isoperimetry}
    {(\mu^{1:N})}^+(A)
    &\deq \lim_{\varepsilon \searrow 0} \frac{\mu^{1:N}(A^\varepsilon) - \mu^{1:N}(A)}{\varepsilon}
    \ge \frac{\mc I(\mu^{1:N}(A))}{L}\,,
\end{align}
where $A^\varepsilon$ denotes the $\varepsilon$-enlargement of $A$.
See~\cite{MikShe23HeatFlow} for further implications.

\paragraph*{Comparison with~\cite{Wang24UnifLSI}.}
Under more general assumptions,~\cite{Wang24UnifLSI} also establishes a log-Sobolev inequality with similar constant. More precisely,~\cite{Wang24UnifLSI} operates under Assumptions~\ref{ass:cvxty} and~\ref{ass:smoothness}, as well as $\norm{\nabla^2 \delta \mc F_0(\nu)} \leq \tilde\beta < \infty$ although we note that $\tilde\beta$ does not appear in the final bound on the LSI constant.
In comparison, we further adopt Assumption~\ref{ass:bddgrad}, as well as the assumption that $\mc F$ has the perturbative structure $\mc F(\nu) = \mc F_0(\nu) + \int V\,\D \mu$.
Furthermore, the result of~\cite{Wang24UnifLSI} is in terms of $\rho \deq 1/C_{\LSI}(\pi)$, whereas the following lower bound on $\rho$
follows from our assumptions as a consequence of Lemma~\ref{lem:lip-pert-lsi}:
\begin{align}\label{eq:lsi_const_pi}
    \frac{1}{\rho} = C_{\LSI}(\pi) \leq \frac{\sigma^2}{2\lambda} \exp \Bigl(\frac{2B^2}{\lambda \sigma^2} + \frac{4\sqrt{2}B}{\sqrt{\lambda}\sigma}\Bigr)\,.
\end{align}

The bound of~\cite{Wang24UnifLSI} on the log-Sobolev constant reads
\begin{align}\label{eq:songbo}
    C_{\LSI}(\mu^{1:N}) \leq \biggl\{\frac{1+2d\,(5+3\,(\varepsilon^{-1} - 1)\,\kappa)\,\frac{\kappa}{1-\kappa/N}} {1-\varepsilon - (8\kappa + 6\,(\varepsilon^{-1} - 1))\,\kappa^2/N} \biggl\}\,\frac{1}{\rho}\,,
\end{align}
for any $\varepsilon > 0$ and $N > \kappa$, where $\kappa \deq \beta/\rho$.
Our result in Theorem~\ref{thm:main-generic}, which however holds for all $N \ge 1$, leads to substantially larger estimates than~\eqref{eq:songbo} when we substitute in~\eqref{eq:lsi_const_pi}.
Under the additional structure of Example~\ref{ex:nn}, our Theorem~\ref{thm:main-specific} has the notable advantage of producing an estimate which is independent of the ambient dimension $d$.

Finally, we note that the stronger consequences of Lipschitz transport maps, such as the Gaussian isoperimetric inequality~\eqref{eq:gaussian_isoperimetry}, seem to be out of reach of the techniques of~\cite{Wang24UnifLSI}.

\section{Technical overview}\label{sec:technical_overview}

\subsection{Proof outline}

The Lipschitz transport map from the standard Gaussian measure $\gamma$ is constructed from the Kim--Milman reverse heat flow~\cite{KimMil12ReverseHeat}, and we rely on recent estimates along this flow from~\cite{brigati2024heat}.
Specifically, if $\mu$ is the measure of interest, then we can consider the curve of measures ${(\mu_t)}_{t\ge 0}$ given by
\begin{align}\label{eq:fokker_planck}
    \partial_t \mu_t = \divergence\Bigl(\mu_t \nabla \log \frac{\D\mu_t}{\D\gamma}\Bigr)\,, \qquad \mu_0 = \mu\,.
\end{align}
On one hand, this equation is the well-known Fokker--Planck equation which describes the evolution of the marginal law of the Ornstein--Uhlenbeck process, and therefore $\mu_t \to \gamma$ as $t\to\infty$.
On the other hand, the equation~\eqref{eq:fokker_planck} is in continuity equation form, and hence $\mu_t = {(S_t)}_\# \mu$ where $S_t$ is the flow map
\begin{align*}
    \partial_t S_t(x) = -\nabla \log \frac{\D\mu_t}{\D\gamma}(S_t(x))\,, \qquad S_0(x) = x\,.
\end{align*}
If $T_t \deq S_t^{-1}$ denotes the inverse flow map, then one could expect that $T_\infty$ is a map from $\gamma$ to $\mu$.
Moreover, the Lipschitz constant of $T_\infty$ can be bounded in terms of estimates for the driving vector fields.
The upshot of this construction is the following condition for the existence of a Lipschitz transport map.

\begin{lemma}[Tilt stability, informal]\label{lem:tilt_informal}
    Let $\mu$ be a probability measure over $\R^d$, and for each $t > 0$ and $y\in\R^d$ we define the probability measure
    \begin{align*}
        \mu_{t,y}(\D x)
        &\propto \exp\Bigl(-\frac{\norm{x-y}^2}{2t} + \frac{\norm x^2}{2}\Bigr)\,\mu(\D x)\,,
    \end{align*}
    where we assume that this does indeed yield a valid probability measure.
    Suppose that we can bound the operator norm of the covariance matrix of $\mu_{t,y}$ as follows: for some $p > 1$,
    \begin{align*}
        \norm{\cov_{\mu_{t,y}}}_{\op}
        &\le \begin{cases}
            t + O(t^p)\,, & \text{for small}~t\,, \\
            O(1)\,, & \text{for large}~t\,,
        \end{cases}
    \end{align*}
    uniformly over $y\in\R^d$.
    Then, there exists a Lipschitz transport map from $\gamma$ to $\mu$.
\end{lemma}

We state a precise quantitative form of this principle as Theorem~\ref{thm:heat_flow}.

Applying this with $\mu$ replaced by $\mu^{1:N}$ and $y$ replaced by $y^{1:N} \in \R^{d\times N}$, our task boils down to bounding the operator norms of the covariance matrices of the tilted measures $\mu_{t,y^{1:N}}$.
Crucially, we want our bounds to be \emph{independent} of the number of particles $N$.

The intuition that we leverage is that if $\mu_{t,y^{1:N}}$ were a product measure, then the covariance bound would indeed be independent of $N$. On the other hand, the phenomenon of \emph{propagation of chaos} quantifies how close $\mu_{t,y^{1:N}}$ is to a product measure.
When implementing this strategy, even when we start with an exchangeable measure $\mu^{1:N}$, the tilting by $y^{1:N}$ causes $\mu_{t,y^{1:N}}$ to lose this property.
To address this, we generalize the propagation of chaos arguments from the literature to the non-exchangeable case, which requires a careful definition of the product measure to which we establish closeness.

This was the strategy carried out previously in~\cite{kook2024sampling}. As explained in the introduction, the heterogeneous propagation of chaos argument carried out therein was based on the argument of~\cite{chen2022uniform} and eventually incurs a doubly exponential dependence for the Lipschitz constant.
This is because the propagation of chaos bounds obtained in this manner depend on the log-Sobolev constant of the mean-field measure $\pi$ (as well as other proximal Gibbs measures), which already scales exponentially---and becomes doubly exponential after passing through the heat flow argument.

In this work, we instead leverage the recent propagation of chaos bound from~\cite{Nit24MeanField}, which controls $\KL(\mu^{1:N} \mmid \pi^{\otimes N})$ without any dependency on the log-Sobolev constant of $\pi$.
We first generalize this result to the non-exchangeable case, which we state as Theorems~\ref{thm:poc} and~\ref{thm:poc-ii}.

However, when we apply the bound to control $\norm{\cov_{\mu_{t,y}}}_{\op}$, via the Donsker--Varadhan principle, it turns out that the bound only scales as $O(t)$ for small $t$, whereas the condition of a $t+o(t)$ bound for small $t$ in Lemma~\ref{lem:tilt_informal} is crucial---otherwise, it leads to divergent integrals in the heat flow estimates.
We remedy this by using a different covariance estimate for small $t$ which enjoys the right scaling, and we stitch together the estimates from small and large times to obtain our final result.

\subsection{Auxiliary results}

We use the following results from prior works.

\begin{lemma}[{\cite[Corollary 2.4]{KhuMaaPed24LInfOT}}]\label{lem:winf}
    Suppose that $\nu$ is $\alpha$-strongly log-concave, and that $\hat\nu \propto \exp(-F)\,\nu$, where $F$ is $L$-Lipschitz.
    Then, $\mc W_\infty(\nu,\hat\nu) \le L/\alpha$.
\end{lemma}

\begin{lemma}[{LSI under Lipschitz perturbations~\cite[Theorem 1.4]{brigati2024heat}}]\label{lem:lip-pert-lsi}
    Let $\mu \propto \exp(-H - V)$, where $V, H: \R^d \to \R$ such that $V$ is $\alpha$-strongly convex and $H$ is $L$-Lipschitz. Then, $\mu$ satisfies a log-Sobolev inequality with constant $C_{\LSI}(\mu)$ given by
    \begin{align*}
        C_{\LSI}(\mu) \leq \frac{1}{\alpha} \exp \Bigl(\frac{L^2}{\alpha} + \frac{4L}{\sqrt{\alpha}} \Bigr)\,.
    \end{align*}
\end{lemma}

We also recall that if $\nu$ is $\alpha$-strongly log-concave, meaning that $\nu \propto \exp(-V)$ for some $\alpha$-strongly convex function $V : \R^d\to\R$, then it satisfies~\eqref{eq:LSI} with $C_{\msf{LSI}}(\nu) \le 1/\alpha$, and moreover, that it satisfies a Poincar\'e inequality
\begin{align*}
    \var_\nu f
    &\le \frac{1}{\alpha}\E_\nu[\norm{\nabla f}^2] \qquad\text{for all compactly supported, smooth}~f : \R^d\to\R
\end{align*}
and Talagrand's $\msf T_2$ inequality
\begin{align*}
    \KL(\nu' \mmid \nu)
    &\ge \frac{\alpha}{2}\,\mc W_2^2(\nu,\nu') \qquad\text{for all}~\nu' \ll \nu\,.
\end{align*}
Also,~\eqref{eq:LSI} implies the following sub-Gaussian concentration inequality: for all $1$-Lipschitz $f : \R^d\to\R$,
\begin{align*}
    \nu\{f \ge \E_\nu f + t\} \le \exp\Bigl(-\frac{t^2}{2C_{\msf{LSI}}(\nu)}\Bigr) \qquad\text{for all}~t > 0\,.
\end{align*}

\section{A generalized propagation of chaos result}

\subsection{Generalized propagation of chaos}

In this section, we prove a propagation of chaos result which is inspired by the argument of~\cite{Nit24MeanField}.

\begin{theorem}[Generalized propagation of chaos]\label{thm:poc}
    For each $i\in [N]$, let $V_i : \R^d\to\R$, and let $\mc F_0 : \mc P_2(\R^d)\to\R$.
    Define the probability measures
    \begin{align*}
        \mu^{1:N}(x^{1:N})
        &\propto \exp\Bigl(-\sum_{i=1}^N V_i(x^i) - \frac{2N}{\sigma^2} \, \mc F_0(\rho_{x^{1:N}})\Bigr)\,, \\
        \pi^i(x^i)
        &\propto \exp\Bigl(-\,V_i(x^i) - \frac{2}{\sigma^2} \,\delta \mc F_0(\bar \pi, x^i)\Bigr)\,,
    \end{align*}
    where $\bar \pi \deq \frac{1}{N} \sum_{i=1}^N \pi^i$, $\pi^{1:N} \deq \bigotimes_{i=1}^N \pi^i$, and we assume that these measures are well-defined.
    Adopt Assumptions~\ref{ass:cvxty},~\ref{ass:smoothness}, and~\ref{ass:bddgrad}, and assume that each $\pi^i$ satisfies a Poincar\'e inequaltiy with constant $\bar {C}_{\msf{PI}}$ and that $V_i$ is $\alpha$-strongly convex.
    Then,
    \begin{align*}
        \KL(\mu^{1:N} \mmid \pi^{1:N})
        &\le \frac{4\betB}{\sigma^2}\min\Bigl\{\bar{C}_{\msf{PI}} d,\, \frac{2d}{\alpha} + \frac{4B^2}{\alpha^2\sigma^4}\Bigr\}\,.
    \end{align*}
\end{theorem}

Before giving the proof, we compare this result to the propagation of chaos bounds from~\cite{chen2022uniform, kook2024sampling}.
(We compare with~\cite{Nit24MeanField} in the next subsection.)

The result of~\cite[Theorem 5]{kook2024sampling} reads as follows: suppose that all of the potentials $V_i$ are the same, $V_i = V$ for $i\in [N]$. Then, for all $N \ge 160\bar\beta \bar {C}_{\msf{LSI}}/\sigma^2$, it holds that $\KL(\mu^{1:N} \mmid \pi^{\otimes N}) \le 33\bar\beta \bar {C}_{\msf{LSI}} d/\sigma^2$, where $\bar\beta$ is a stronger notion of smoothness (c.f.\ Remark~\ref{rmk:smoothness}).
We improve their result along several axes: (1) our result allows for heterogeneous environments (i.e., the $V_i$ are allowed to differ), which is crucial for the heat flow estimates in \S\ref{sec:heat_flow}; (2) we replace the dependence on the log-Sobolev constant with a dependence on the Poincar\'e constant; (3) our bound holds for all $N \ge 1$; (4) we provide a second estimate which does not depend on the Poincar\'e constant, which is important for settings in which $\bar {C}_{\msf{PI}}$ is exponentially large.

We now turn toward the proof of Theorem~\ref{thm:poc}.

\medskip{}

\begin{proof}
Let us define the Bregman divergence as follows:
\begin{equation*}
    B_{\mathcal{F}_0}(\rho_{x^{1:N}}, \bar{\pi})
    \deq \mathcal{F}_0(\rho_{x^{1:N}}) - \mathcal{F}_0(\bar{\pi}) 
    - \langle \delta \mathcal{F}_0(\bar{\pi}), \rho_{x^{1:N}} - \bar{\pi} \rangle\,.
\end{equation*}
Because of the convexity of $\mathcal{F}_0$, $B_{\mathcal{F}_0}(\rho_{x^{1:N}}, \bar{\pi}) \geq 0$. It follows that
\begin{align*}
    \mu^{1:N}(x^{1:N}) 
    &\propto \exp\Bigl( - \sum_{i=1}^N V_i(x^i) 
    - \frac{2N}{\sigma^2}\,\mathcal{F}_0(\rho_{x^{1:N}}) \Bigr) \\
    &\propto \exp\Bigl( - \sum_{i=1}^N V_i(x^i) 
    - \frac{2N}{\sigma^2}\,\bigl( B_{\mathcal{F}_0}(\rho_{x^{1:N}}, \bar{\pi}) + \langle \delta \mathcal{F}_0(\bar{\pi}), \rho_{x^{1:N}} \rangle\bigr) \Bigr) \\
    &\propto \exp\Bigl( - \sum_{i=1}^N V_i(x^i) 
    - \frac{2}{\sigma^2}\sum_{i=1}^N \delta \mathcal{F}_0(\bar{\pi},x^i)
    - \frac{2N}{\sigma^2}\, B_{\mathcal{F}_0}(\rho_{x^{1:N}}, \bar{\pi}) \Bigr) \\
    &\propto \exp\Bigl( - \frac{2N}{\sigma^2}\, B_{\mathcal{F}_0}(\rho_{x^{1:N}}, \bar{\pi}) \Bigr)\, \pi^{1:N}(x^{1:N})\,.
\end{align*}
Let $Z$ denote the normalization constant for the right-hand side: 
\begin{equation*}
    Z \deq \int \exp\Bigl( - \frac{2N}{\sigma^2}\, B_{\mathcal{F}_0}(\rho_{x^{1:N}}, \bar{\pi}) \Bigr)\, \pi^{1:N}( \mathrm{d}x^{1:N})\,.
\end{equation*}
By Jensen's inequality, we see that $\log Z \geq - \frac{2N}{\sigma^2} \int B_{\mathcal{F}_0}(\rho_{x^{1:N}}, \bar{\pi})\, \pi^{1:N}( \mathrm{d}x^{1:N})$.
Therefore, we get
\begin{align*}
    \KL(\mu^{1:N} \mmid \pi^{1:N})
    &= \int \mu^{1:N}(\mathrm{d}x^{1:N}) \log \frac{\exp\bigl( - \frac{2N}{\sigma^2}\, B_{\mathcal{F}_0}(\rho_{x^{1:N}}, \bar{\pi}) \bigr)}{Z} \\
    &= - \frac{2N}{\sigma^2} \int B_{\mathcal{F}_0}(\rho_{x^{1:N}}, \bar{\pi})\, \mu^{1:N}(\mathrm{d}x^{1:N}) 
    - \log Z \\
    &\leq \frac{2N}{\sigma^2} \int B_{\mathcal{F}_0}(\rho_{x^{1:N}}, \bar{\pi})\, \pi^{1:N}( \mathrm{d}x^{1:N})\,.
\end{align*}
By convexity of $\mc F_0$,
\begin{align*}
    \E_{x^{1:N}\sim \pi^{1:N}} B_{\mc F_0}(\rho_{x^{1:N}},\bar\pi)
    &= \E_{x^{1:N}\sim \pi^{1:N}}[\mc F_0(\rho_{x^{1:N}}) - \mc F_0(\bar \pi) - \langle \delta \mc F_0(\bar \pi), \rho_{x^{1:N}} - \bar\pi\rangle] \\
    &=\E_{x^{1:N}\sim \pi^{1:N}}[\mc F_0(\rho_{x^{1:N}}) - \mc F_0(\bar \pi)]
    \le \E_{x^{1:N}\sim\pi^{1:N}}\langle \delta \mc F_0(\rho_{x^{1:N}}), \rho_{x^{1:N}} - \bar\pi\rangle \\
    &= \frac{1}{N} \sum_{i=1}^N \E_{x^{1:N}\sim\pi^{1:N}}\langle \delta \mc F_0(\rho_{x^{1:N}}), \delta_{x^i} - \pi^i\rangle \\
    &= \frac{1}{N} \sum_{i=1}^N \E_{x^{1:N}\sim\pi^{1:N}} \int [\delta \mc F_0(\rho_{x^{1:N}}, x^i) - \delta \mc F_0(\rho_{x^{1:N}}, z^i)]\,\pi^i(\D z^i)\,.
\end{align*}
Next, let $\tilde x^i \sim \pi^i$ be drawn independently from $x^{1:N} \sim \pi^{1:N}$, and let $\tilde x_i^{1:N} \deq (x^1,\dotsc,x^{i-1},\tilde x^i, x^{i+1},\dotsc,x^N)$.
We can write
\begin{align*}
    &\E_{x^{1:N}\sim \pi^{1:N}} B_{\mc F_0}(\rho_{x^{1:N}},\bar\pi)
    \le \frac{1}{N} \sum_{i=1}^N \underbrace{\E_{\substack{x^{1:N}\sim\pi^{1:N} \\ \tilde x^i \sim \pi^i}} \int [\delta \mc F_0(\rho_{\tilde x_i^{1:N}}, x^i) - \delta \mc F_0(\rho_{\tilde x_i^{1:N}}, z^i)]\,\pi^i(\D z^i)}_{=0} \\
    &\qquad\qquad{} + \frac{1}{N} \sum_{i=1}^N \E_{\substack{x^{1:N}\sim \pi^{1:N} \\ \tilde x^i \sim \pi^i}} \int_0^1\iint [\delta^2 \mc F_0(\rho_t, x^i, y^i) - \delta^2 \mc F_0(\rho_t, z^i, y^i)]\,\pi^i(\D z^i)\,(\rho_{x^{1:N}} - \rho_{\tilde x_i^{1:N}})(\D y^i)\,\D t \\
    &\qquad = \frac{1}{N^2} \sum_{i=1}^N \E_{\substack{x^{1:N}\sim \pi^{1:N} \\ \tilde x^i \sim \pi^i}} \int_0^1\iint [\delta^2 \mc F_0(\rho_t, x^i, y^i) - \delta^2 \mc F_0(\rho_t, z^i, y^i)]\,\pi^i(\D z^i)\,(\delta_{x^i} - \delta_{\tilde x^i})(\D y^i)\,\D t\,,
\end{align*}
where $\rho_t \deq (1-t)\,\rho_{\tilde x_i^{1:N}} + t\,\rho_{x^{1:N}}$.
We can further write this as
\begin{align*}
    &\E_{x^{1:N}\sim \pi^{1:N}} B_{\mc F_0}(\rho_{x^{1:N}},\bar\pi) \\
    &\qquad \le \frac{1}{N^2} \sum_{i=1}^N \E_{\substack{x^{1:N}\sim \pi^{1:N} \\ \tilde x^i \sim \pi^i}} \int_0^1\int_0^1\iint \langle \nabla_1\delta^2 \mc F_0(\rho_t, z_u^i, y^i), x^i - z^i\rangle \,\pi^i(\D z^i)\,(\delta_{x^i} - \delta_{\tilde x^i})(\D y^i)\,\D t\,\D u \\
    &\qquad \le \frac{1}{N^2} \sum_{i=1}^N \E_{\substack{x^{1:N}\sim \pi^{1:N} \\ \tilde x^i \sim \pi^i}} \int_0^1\int_0^1\int_0^1\int \langle \nabla_1\nabla_2 \delta^2 \mc F_0(\rho_t, z_u^i, \tilde x_v^i)\, (x^i - z^i), x^i - \tilde x^i\rangle \,\pi^i(\D z^i)\,\D t\,\D u\,\D v \\
    &\qquad \le \frac{\betB}{N^2} \sum_{i=1}^N \E_{\substack{x^{1:N}\sim \pi^{1:N} \\ \tilde x^i \sim \pi^i}}\int \norm{x^i - z^i}\,\norm{x^i - \tilde x^i} \,\pi^i(\D z^i)\,,
\end{align*}
where $z^i_u \deq (1-u)\,z_i + u\,x^i$ and $\tilde x^i_v \deq (1-v)\,\tilde x^i + v\,x^i$.
By the Cauchy--Schwarz inequality,
\begin{align*}
    \E_{x^{1:N}\sim \pi^{1:N}} B_{\mc F_0}(\rho_{x^{1:N}},\bar\pi)
    &\le \frac{\betB}{N^2} \sum_{i=1}^N \iint \norm{x^i - z^i}^2\,\pi^i(\D x^i)\,\pi^i(\D z^i)\,.
\end{align*}

We now provide two bounds on this term which are effective when the Poincar\'e constant for $\pi^{1:N}$ is small or large respectively.
In the first case, for $X^i \sim \pi^i$,
\begin{align*}
    \iint \norm{x^i - z^i}^2 \,\pi^i(\D x^i)\,\pi^i(\D z^i)
    &= 2\E[\norm{X^i - \E X^i}^2]
    \le 2\bar {C}_{\msf{PI}} d\,,
\end{align*}
by the Poincar\'e inequality.
In the second case, let $\breve \pi^i \propto \exp(-V_i)$.
Since $\breve\pi^i$ is $\alpha$-strongly log-concave and $\delta \mc F_0(\bar \pi)$ is $B$-Lipschitz by Assumption~\ref{ass:bddgrad}, we may apply Lemma~\ref{lem:winf} to obtain
\begin{align*}
    2\E[\norm{X^i - \E X^i}^2]
    \le 2\E[\norm{X^i - \E\breve X^i}^2]
    \le 4\E[\norm{\breve X^i - \E\breve X^i}^2] + 4\mc W_\infty^2(\pi^i, \breve\pi^i)
    \le \frac{4d}{\alpha} + \frac{8B^2}{\alpha^2\sigma^4}\,,
\end{align*}
where $\breve X^i \sim \breve \pi^i$ and the bound on the covariance of $\breve\pi^i$ follows from the Brascamp--Lieb inequality~\cite{BraLie1976}.
\end{proof}

\subsection{Refined bound}

The general estimate in Theorem~\ref{thm:poc} incurs a dependence on the dimension $d$ of each particle.
In the next theorem, we prove a dimension-free estimate when $\mc F_0$ further has the structure given in Example~\ref{ex:nn}.

The result can be compared to~\cite[Theorem 1]{Nit24MeanField}: in that result, it is assumed that $V_i = V$ for all $i\in [N]$, and that $h$ is bounded by a constant $R_h$, and the resulting bound is $\KL(\mu^{1:N} \mmid \pi^{\otimes N}) \le 4\beta_\ell R_h^2/\sigma^2$.
The result below considers the the more general setting of heterogeneous environments and replaces the assumption of boundedness of $h$ with the assumption of Lipschitzness of $h$, which is more realistic for the application to two-layer networks.
Moreover, the bound below improves when $\pi^{1:N}$ has better isoperimetric properties, which is needed for the heat flow estimates in \S\ref{sec:heat_flow}.

\begin{theorem}[Generalized propagation of chaos II]\label{thm:poc-ii}
    Consider the setting of Theorem~\ref{thm:poc} but further assume that $\mc F_0$ has the structure in Example~\ref{ex:nn}.
    Then,
    \begin{align*}
        \KL(\mu^{1:N} \mmid \pi^{1:N})
        &\le \frac{\betB}{\sigma^2} \min\Bigl\{\bar {C}_{\msf{PI}}, \, \frac{2}{\alpha} + \frac{8B^2}{\alpha^2\sigma^4}\Bigr\}\,.
    \end{align*}
\end{theorem}
\begin{proof}
    We follow the proof of Theorem~\ref{thm:poc} until the bound on $\E_{x^{1:N}\sim\pi^{1:N}} B_{\mc F_0}(\rho_{x^{1:N}},\bar\pi)$.
    Here,
    \begin{align*}
        \E_{x^{1:N}\sim\pi^{1:N}} B_{\mc F_0}(\rho_{x^{1:N}},\bar\pi)
        &= \int \E_{x^{1:N}\sim\pi^{1:N}}\bigl[\ell\bigl(\E_{\rho_{x^{1:N}}} h(\cdot,z),z\bigr)-\ell\bigl(\E_{\bar\pi} h(\cdot,z),z\bigr)\bigr] \, P(\D z)\,.
    \end{align*}
    We subsequently bound the integrand pointwise, and to alleviate notation, we therefore drop the dependence on $z$.
    Continuing,
    \begin{align*}
        (\text{integrand})
        &\le \mathbb{E}_{x^{1:N} \sim \pi^{1:N}} \Bigl[ \ell'(\mathbb{E}_{\bar\pi} h)\,
        ( \mathbb{E}_{\rho_{x^{1:N}}} h - \mathbb{E}_{\bar{\pi}} h) 
        + \frac{\beta_\ell}{2}\,{( \mathbb{E}_{\rho_{x^{1:N}}} h - \mathbb{E}_{\bar{\pi}} h)}^2 \Bigr] \\
        &= \frac{\beta_\ell}{2}\E_{x^{1:N}\sim\pi^{1:N}}\bigl[{( \mathbb{E}_{\rho_{x^{1:N}}} h - \mathbb{E}_{\bar{\pi}} h)}^2 \bigr]
        = \frac{\beta_\ell}{2}\E_{x^{1:N}\sim \pi^{1:N}}\Bigl[\Bigl( \frac{1}{N} \sum_{i=1}^N \bigl(h(x^i)- \E_{\pi^i} h\bigr)\Bigr)^2 \Bigr] \\
        &= \frac{\beta_\ell}{2N^2} \sum_{i=1}^N \var_{\pi^i}(h)\,.
    \end{align*}
    As before, we provide two estimates: one that depends on the Poincar\'e constant and one that does not.
    The first estimate simply applies the Poincar\'e inequality:
    \begin{align*}
        \var_{\pi^i}(h)
        &\le \bar {C}_{\msf{PI}} \E_{\pi^i}[\norm{\nabla h}^2]
        \le \bar {C}_{\msf{PI}} L_h^2\,.
    \end{align*}
    The second estimate uses Lemma~\ref{lem:winf}:
    \begin{align*}
        \var_{\pi^i}(h)
        &\le \E_{\pi^i}[{(h-\E_{\breve \pi^i} h)}^2]
        \le 2\E_{\breve\pi^i}[{(h-\E_{\breve \pi^i} h)}^2] + 2L_h^2\mc W_\infty^2(\pi^i, \breve\pi^i)
        \le \frac{2L_h^2}{\alpha} + \frac{8B^2 L_h^2}{\alpha^2\sigma^4}\,,
    \end{align*}
    where we used the Poincar\'e inequality for $\breve \pi^i$ (which follows from strong log-concavity).
    The final bound follows by recalling from Example~\ref{ex:nn} that $\betB = L_h^2 \beta_\ell$.
\end{proof}

\section{Reverse heat flow estimates}\label{sec:heat_flow}

We first state a precise version of the Lipschitz estimate in terms of the tilt stability condition.
Note that the following theorem recovers the bound in~\cite[Theorem 1.4]{brigati2024heat} for a particular choice of $(C_m, k_m)_{m \in [M]}$.
Let ${(P_t)}_{t\ge 0}$, ${(Q_t)}_{t\ge 0}$ denote the heat and Ornstein--Uhlenbeck semigroups respectively.

\begin{theorem}[Lipschitz transport maps via tilt stability]\label{thm:heat_flow}
    Assume that the following bound holds for some $a, \zeta > 0$ and a sequence $(C_m, k_m)_{m \in[M]}$ with $k_m > 1$ for all $m \in [M]$:
    \begin{align*}
        \norm{\cov_{\mu_{t,y}}}_{\op} \leq \frac{1}{a+ 1/t} + \sum_{m=1}^M \frac{C_m}{{(a+1/t)}^{k_m}}\qquad \text{for all}~y \in \R^d,\, t > 0\,.
    \end{align*}
    Then, there exists an $L$-Lipschitz transport map $T: \R^d \to \R^d$ such that $T_\# \gamma = \mu$, where $\gamma$ is the standard Gaussian measure and $L$ can be estimated by 
    \begin{align*}
        L \leq \frac{1}{\sqrt{a+1}} \exp \biggl(\sum_{m=1}^M \frac{C_m}{2\,(k_m -1)\, {(a+1)}^{k_m - 1}} \biggr)\,.
    \end{align*}
\end{theorem}
\begin{proof}
    Following the calculations of~\cite{brigati2024heat},
    \begin{align*}
        -\frac{1}{\exp(2t) - 1}
        &\preceq \nabla^2 \log Q_t \Big(\frac{\mu}{\gamma}\Big) = \exp(-2t) \Big[\nabla^2 \log P_{1-\exp(-2t)} \frac{\mu}{\gamma} \Big](\exp(-t) \cdot)\\
        &= \frac{1}{\exp(2t)-1}\,\Bigl( \frac{\cov_{\mu_{1-\exp(-2t),\,\exp(-t)\,\cdot}}}{1-\exp(-2t)} - I \Bigr) \\
        &\preceq \biggl[\frac{1-\alpha}{\alpha\,(\exp(2t)-1)+1} + \sum_{m=1}^M \frac{C_m \exp(2t) \,{(\exp(2t)-1)}^{k_m-2}}{{(\alpha\,(\exp(2t)-1)+1)}^{k_m}}\biggr]\,I\,,
    \end{align*}
    where we take $\alpha = a+1$. We proceed to integrate the upper bound from $t=0$ to $t=\infty$. Note that, apart from the very first term, the generic form of the heat flow integral is
    \begin{align*}
        &\int_0^\infty \frac{\exp(2t)\,{(\exp(2t)-1)}^{k_m-2}}{{(\alpha\,(\exp(2t)-1) + 1)}^{k_m}} \, \D t  = \int_0^\infty \frac{(\tau+1)\,\tau^{k_m-2}}{{(\alpha \tau+1)}^{k_m}}\, \frac{1}{2\,(\tau+1)} \, \D \tau = \frac{1}{2\,(k_m-1)\,\alpha^{k_m-1}}\,.
    \end{align*}
    Here, we make the change of variables $\tau = \exp(2t)-1$.
    
    We now deal with the remaining term.
    \begin{align*}
        \int_0^{\infty} &\frac{1-\alpha}{\alpha\,(\exp(2t)-1)+1} \, \D t =\int_0^{\infty} \frac{1-\alpha}{\alpha \tau + 1}\, \frac{1}{2\,(\tau+1)} \, \D \tau = -\frac{1}{2} \log \alpha\,.
    \end{align*}
    The result follows from~\cite[Lemma 3.1]{brigati2024heat}.
\end{proof}

For $y^1,\dotsc,y^N \in \R^d$, let
\begin{align*}
    \mu_{t, y^{1:N}}(\D x^{1:N})
    &\propto \exp\Bigl(-\frac{\norm{y^{1:N}-x^{1:N}}^2}{2t} + \frac{\norm{x^{1:N}}^2}{2}\Bigr)\,\mu^{1:N}(\D x^{1:N})\,,\\
    \pi_{t, y^i}(\D x^i)
    &\propto \exp\Bigl(-\frac{\norm{y^i-x^i}^2}{2t} + \frac{\norm{x^i}^2}{2} - \frac{2}{\sigma^2} \,V(x^i) - \frac{2}{\sigma^2}\,\delta \mc F_0(\bar \pi_{t,y^{1:N}},x^i)\Bigr)\,\D x^i\,,
\end{align*}
where $\mu^{1:N}$ is defined in~\eqref{eq:finite-particle-measure} and we denote $\pi_{t,y^{1:N}} \deq \bigotimes_{i=1}^N \pi_{t,y^i}$ and $\bar \pi_{t,y^{1:N}} \deq \frac{1}{N} \sum_{i=1}^N \pi_{t,y^i}$.
We adopt the assumptions in \S\ref{sec:main_result}, and in particular, these define valid probability measures as soon as $2\lambda/\sigma^2 > 1$.
(Note that the measures $\pi_{t,y^i}$, $i\in [N]$ solve a system of (implicit) equations, and the uniqueness of the solution is argued in~\cite[Lemma 24]{kook2024sampling}.)

We further introduce
\begin{align*}
    \breve\pi_{t,y^i}(\D x^i)
    &\propto \exp\Bigl(-\frac{\norm{y^i-x^i}^2}{2t} + \frac{\norm{x^i}^2}{2} - \frac{2}{\sigma^2}\,V(x^i)\Bigr)\,\D x^i\,, \qquad \breve \pi_{t,y^{1:N}} \deq \bigotimes_{i=1}^N \breve \pi_{t,y^i}\,,
\end{align*}
so that $\pi_{t,y^i}(\D x^i) \propto \exp(-\frac{2}{\sigma^2}\,\delta \mc F_0(\bar\pi_{t,y^{1:N}},x^i))\,\breve\pi_{t,y^i}(\D x^i)$.

We now state two lemmas which bound the covariance of $\mu_{t, y^{1:N}}$ in different regimes.
Throughout, we use the shorthand $\alpha_t \deq 2\lambda/\sigma^2 - 1 + 1/t$.

\begin{lemma}[Small $t$ regime]\label{lem:small-regime}
    Suppose $t \leq {(20B^2/\sigma^4 - 2\lambda/\sigma^2 + 1)}^{-1}$. Then, we can bound
    \begin{align*}
        \norm{\cov \mu_{t, y^{1:N}}}_{\op} \leq \Bigl[\frac{1}{\sqrt{\alpha_t}} + O\Bigl(\frac{1}{\alpha_t}\,\bigl(\frac{\sqrt{\betB \dprox}}{\sigma} + \frac{B}{\sigma^2} \bigr)\Bigr)\Bigr]^2\,.
    \end{align*}
    Here,
    \begin{align*}
        \dprox \deq \begin{cases}
            d\,, & \text{under Theorem~\ref{thm:poc}}\,,\\
            1\,, & \text{under Theorem~\ref{thm:poc-ii}}\,.
        \end{cases}
    \end{align*}
\end{lemma}
\begin{proof}
   Note that uniformly over all $y^i \in \R^d$, $\pi_{t, y^i}$ satisfies the Poincar\'e and log-Sobolev inequalities with constants $\bar C_{\msf{PI}}$, $\bar C_{\LSI}$ given respectively by
    \begin{align*}
        \bar C_{\msf{PI}}
        \le \bar C_{\msf{LSI}}
        &\le \frac{1}{\alpha_t} \exp\Bigl(\frac{4B^2/\sigma^4}{\alpha_t} + \frac{8B/\sigma^2}{\sqrt{\alpha_t}}\Bigr)
        \le \frac{3}{\alpha_t} \exp\Bigl(\frac{20B^2/\sigma^4}{\alpha_t}\Bigr)\,.
    \end{align*}
    This bound is obtained by applying Lemma~\ref{lem:lip-pert-lsi} to a $2B/\sigma^2$-Lipschitz perturbation of an $\alpha_t$-strongly log-concave measure.
    
    Note that in the given regime of $t$, the exponential term can further be bounded by $e$. Thus, we have
    \begin{align*}
        \mc W_2(\mu_{t, y^{1:N}}, \pi_{t, y^{1:N}})
        &\le \sqrt{2\bar {C}_{\LSI} \KL(\mu_{t, y^{1:N}} \mmid \pi_{t, y^{1:N}})}
        \lesssim \frac{1}{\sqrt{\alpha_t}} \exp\Bigl(\frac{10B^2/\sigma^4}{\alpha_t} \Bigr)\,\frac{\sqrt{\bar C_{\msf{PI}} \beta \dprox}}{\sigma}
        \lesssim \frac{\sqrt{\beta\dprox}}{\alpha_t \sigma}\,,
    \end{align*}
    using the improved propagation of chaos bounds in Theorems~\ref{thm:poc} and~\ref{thm:poc-ii}. In the first inequality, we applied Talagrand's $\msf T_2$ inequality for $\pi_{t, y^{1:N}}$. Here $\dprox = d$ in the setting of Theorem~\ref{thm:poc}, and $\dprox = 1$ in the setting of Theorem~\ref{thm:poc-ii}.

    Transporting from $\pi_{t, y^i}$ to $\breve \pi_{t, y^i}$, we also obtain
    \begin{align*}
        \norm{\cov \pi_{t, y^i}}_{\op}
        \le \bigl(\sqrt{\norm{\cov\breve \pi_{t,y^i}}_{\op}} + \mc W_2(\pi_{t,y^i}, \breve \pi_{t,y^i})\bigr)^2
        \leq \Bigl(\frac{1}{\sqrt{\alpha_t}} + \frac{2B/\sigma^2}{\alpha_t} \Bigr)^2\,,
    \end{align*}
    by Lemma~\ref{lem:winf} and the Brascamp--Lieb inequality~\cite{BraLie1976}.
    Substitute this into the following bound for $\mu_{t,y^{1:N}}$,
        \begin{align*}
        \norm{\cov \mu_{t, y^{1:N}}}_{\op} \le \Bigl( \max_{i\in [N]} \sqrt{\norm{\cov_{\pi_{t,y^i}}}_{\op}} + \mc W_2(\mu_{t,y^{1:N}}, \pi_{t,y^{1:N}}) \Bigr)^2\,,
    \end{align*}
    where we used the fact that $\pi_{t,y^{1:N}}$ is a product measure.
    This concludes the proof.
\end{proof}

\begin{lemma}[Large $t$ regime]\label{lem:large-regime}
    Let $\dprox$ be as in Lemma~\ref{lem:small-regime}. Suppose $t > {(20B^2/\sigma^4 - 2\lambda/\sigma^2 + 1)}^{-1}$. Then, we have the bound    
    \begin{align*}
        \norm{\cov \mu_{t, y^{1:N}}}_{\op}
        &\le \Bigl[\frac{1}{\sqrt{\alpha_t}} + O\bigl(\frac{B}{\alpha_t \sigma^2}\bigr)\Bigr]^2 + O\Bigl(\frac{\beta B^2 \dprox}{\alpha_t^3 \sigma^6} + \frac{\beta B^4}{\alpha_t^4 \sigma^{10}}\Bigr)\,.
    \end{align*}
\end{lemma}
\begin{proof}
    In this regime, we want to use a tighter bound which does not pass through the log-Sobolev constant $\bar {C}_{\LSI}$. In particular, we do not want to rely on Talagrand's inequality.

    We let $\tilde X^{1:N} \sim \mu_{t,y^{1:N}}$ and $X^{1:N} \sim \pi_{t,y^{1:N}}$.
    Also, let $\breve X^{1:N} \sim \breve \pi_{t,y^{1:N}}$ be optimally coupled to $X^{1:N}$ in the $\mc W_\infty$ metric.
    Applying the Donsker--Varadhan variational principle, for any $\xi > 0$,
    \begin{align*}
        &\langle \theta^{1:N}, \cov_{\mu_{t,y^{1:N}}} \theta^{1:N}\rangle
        \le \E[\langle \theta^{1:N}, \tilde X^{1:N} - \E X^{1:N} \rangle^2] \\
        &\qquad \leq \frac{1}{\xi}\, \Bigl(\KL(\mu_{t,y^{1:N}} \mmid \pi_{t, y^{1:N}}) + \log \E \exp \bigl(\xi\,\langle \theta^{1:N}, X^{1:N} - \E X^{1:N} \rangle^2\bigr)\Bigr) \\
        &\qquad \leq \frac{1}{\xi}\, \Bigl(\KL(\mu_{t,y^{1:N}} \mmid \pi_{t, y^{1:N}}) \\
        &\qquad\qquad\qquad{} + \log \E\exp\bigl(2\xi\,\langle \theta^{1:N}, \breve X^{1:N} - \E \breve X^{1:N} \rangle^2 + 2\xi\,\langle \theta^{1:N}, X^{1:N} - \E X^{1:N} - \breve X^{1:N} + \E \breve X^{1:N}\rangle^2\bigr)\Bigr) \\
        &\qquad\le \frac{1}{\xi}\, \Bigl(\KL(\mu_{t,y^{1:N}} \mmid \pi_{t, y^{1:N}}) + \frac{1}{2} \log \E\exp\bigl(4\xi\,\langle \theta^{1:N}, \breve X^{1:N} - \E \breve X^{1:N} \rangle^2\bigr) \\
        &\qquad\qquad\qquad{}+\frac{1}{2} \log \E\exp\bigl(4\xi\,\langle \theta^{1:N}, X^{1:N} - \E X^{1:N} - \breve X^{1:N} + \E \breve X^{1:N}\rangle^2\bigr)\Bigr)\,.
    \end{align*}
    By~\cite[Proposition 2.6.1]{Ver18HighDimProb},
    \begin{align*}
        \norm{\langle \theta^{1:N}, \breve X^{1:N} - \E \breve X^{1:N}\rangle}_{\psi_2}^2
        &= \Bigl\lVert \sum_{i=1}^N \langle \theta^i, \breve X^i - \E \breve X^i\rangle\Bigr\rVert_{\psi_2}^2
        \lesssim \sum_{i=1}^N {\norm{\langle \theta^i, \breve X^i - \E \breve X^i\rangle}_{\psi_2}^2}
        \lesssim \frac{1}{\alpha_t} \sum_{i=1}^N {\norm{\theta^i}^2\,}
        = \frac{1}{\alpha_t}\,,
    \end{align*}
    using the fact that $\breve \pi_{t,y^i}$ is $\alpha_t$-strongly log-concave, hence we have sub-Gaussian concentration of Lipschitz functions.
    Also,
    \begin{align*}
        \norm{\langle \theta^{1:N}, X^{1:N} - \E X^{1:N} - \breve X^{1:N} + \E \breve X^{1:N}\rangle}_{\psi_2}^2
        &= \Bigl\lVert \sum_{i=1}^N \langle \theta^i, X^i - \E X^i - \breve X^i + \E \breve X^i\rangle\Bigr\rVert_{\psi_2}^2 \\
        &\lesssim \sum_{i=1}^N {\norm{\langle \theta^i, X^i - \E X^i - \breve X^i + \E \breve X^i\rangle}_{\psi_2}^2} \\
        &\lesssim \frac{B^2}{\alpha_t^2 \sigma^4} \sum_{i=1}^N {\norm{\theta^i}^2}
        = \frac{B^2}{\alpha_t^2 \sigma^4}\,,
    \end{align*}
    where we applied Lemma~\ref{lem:winf} and Hoeffding's lemma.

    By~\cite[Proposition 2.7.1, Lemma 2.7.6, and Exercise 2.7.10]{Ver18HighDimProb}, there is a universal constant $c > 0$ such that if $\norm Z_{\psi_2} \le K$ and $Z$ is centered, then
    \begin{align*}
        \log \E \exp(\xi Z^2) - \xi \E[Z^2]
        &\lesssim \xi^2\,\norm Z_{\psi_2}^4\,, \qquad\text{for all}~\abs\xi < \frac{c}{\norm Z_{\psi_2}^2}\,.
    \end{align*}
    Applying this fact, we see that for $\xi \lesssim \alpha_t \wedge (\alpha_t^2 \sigma^4/B^2)$,
    \begin{align*}
        \langle \theta^{1:N}, \cov_{\mu_{t,y^{1:N}}} \theta^{1:N}\rangle
        &\lesssim \frac{1}{\xi} \KL(\mu_{t,y^{1:N}} \mmid \pi_{t, y^{1:N}}) + \frac{1}{\alpha_t} + \frac{\xi}{\alpha_t^2} + \frac{B^2}{\alpha_t^2 \sigma^4} + \frac{\xi B^4}{\alpha_t^4 \sigma^8}\,.
    \end{align*}
    Choosing $\xi$ to be its maximum allowable value, keeping in mind that we are in the regime of $t$ for which $\alpha_t \lesssim B^2/\sigma^4$, it yields
    \begin{align*}
        \norm{\cov_{\mu_{t,y^{1:N}}}}_{\op}
        &\lesssim \frac{B^2}{\alpha_t^2\sigma^4} \KL(\mu_{t,y^{1:N}} \mmid \pi_{t, y^{1:N}}) + \frac{1}{\alpha_t} + \frac{B^2}{\alpha_t^2\sigma^4}
        \lesssim \frac{\beta B^2 \dprox}{\alpha_t^3 \sigma^6} + \frac{\beta B^4}{\alpha_t^4 \sigma^{10}} + \frac{1}{\alpha_t} + \frac{B^2}{\alpha_t^2 \sigma^4}\,,
    \end{align*}
    where we plugged in the propagation of chaos bounds from Theorems~\ref{thm:poc} and~\ref{thm:poc-ii}.
    On the other hand,
    \begin{align*}
        \Bigl(\frac{1}{\sqrt{\alpha_t}} + \frac{CB}{\alpha_t\sigma^2}\Bigr)^2
        &\ge \frac{C^2 B^2}{\alpha_t^2\sigma^4}
        \gtrsim \frac{C^2 B^2}{\alpha_t^2 \sigma^4} + \frac{C^2}{\alpha_t}\,,
    \end{align*}
    where the last inequality follows from the regime of $t$.
    Hence, if we choose $C$ large enough, it dominates the last two terms above and yields the final bound.
\end{proof}

We can now prove our main theorems.
\medskip{}

\begin{proof}[Proof of Theorems~\ref{thm:main-generic} and~\ref{thm:main-specific}]
    First, we must ensure that $2\lambda > \sigma^2$.
    Consider the rescaling map $x \mapsto \eta x$, which produces a new measure $\mu_\eta^{1:N} \deq \eta_\# \mu^{1:N}$, given by
    \begin{align*}
        \mu_\eta^{1:N}(x^{1:N}) \propto \exp\Bigl(-\frac{2N}{\sigma^2}\, \mc F_0^\eta(\rho_{x^{1:N}}) - \frac{\lambda}{\eta^2\sigma^2}\, \norm{x^{1:N}}^2 \Bigr)\,.
    \end{align*}
    Here, $\mc F_0^\eta(\nu)  \deq  \mc F_0((\eta^{-1})_{\#} \nu)$, and we choose $\eta = \sqrt\lambda/\sigma$. The rescaled measure satisfies Assumptions~\ref{ass:cvxty},~\ref{ass:smoothness}, and~\ref{ass:bddgrad}, but with parameters $\beta \gets \beta \sigma^2/\lambda$, $\lambda \gets \sigma^2$, $B \gets B\sigma/\sqrt{\lambda}$.
    From Theorem~\ref{thm:heat_flow}, we take $a=1$ and
    \begin{align*}
        C_1 &\lesssim \sqrt{\frac{\beta \dprox}{\lambda}} + \frac{B}{\sqrt\lambda\sigma}\,, & k_1 &= \frac{3}{2}\,, \\[0.25em]
        C_2 &\lesssim \frac{\beta \dprox}{\lambda} + \frac{B^2}{\lambda\sigma^2}\,, & k_2 &= 2\,, \\[0.25em]
        C_3 &\lesssim \frac{\beta B^2 \dprox}{\lambda^2 \sigma^2}\,, & k_3 &= 3\,, \\[0.25em]
        C_4 &= \frac{\beta B^4}{\lambda^3 \sigma^{4}}\,, & k_4 &= 4 \,.
    \end{align*}
    Thus, we have an $L_\eta$-Lipschitz transport map to $\mu_\eta^{1:N}$ with $L_\eta$ given by
    \begin{align*}
        L_\eta
        &\lesssim \exp\Bigl(\mc O\Bigl(\sqrt{\frac{\beta \dprox}{\lambda}} + \frac{B}{\sqrt\lambda\sigma} + \frac{\beta \dprox}{\lambda} + \frac{B^2}{\lambda\sigma^2} + \frac{\beta B^2 \dprox}{\lambda^2 \sigma^2} + \frac{\beta B^4}{\lambda^3 \sigma^{4}}\Bigr)\Bigr) \\
        &\lesssim \exp\Bigl(\mc{O}\Bigl(\frac{\beta \dprox}{\lambda} +\frac{B^2}{\lambda\sigma^2} + \frac{\beta B^2 \dprox}{\lambda^2\sigma^2} + \frac{\beta B^4}{\lambda^3\sigma^{4}} \Bigr)\Bigr)\,.
    \end{align*}
    Finally, composing with the scaling map $x\mapsto \eta^{-1} x$ yields the estimate for $L$.
    Note that when $\dprox = 1$, the third term in the exponential can be controlled in terms of the first and last terms by the Cauchy--Schwarz inequality, so we can omit it to simplify the final bound.
\end{proof}

\section*{Acknowledgments}

We thank Zhenjie Ren, Taiji Suzuki, and Songbo Wang for helpful conversations. MSZ was supported by NSERC through the CGS-D program.

\printbibliography

\end{document}

%% file: _macros.tex
\usepackage{comment,url,algorithm,algorithmic,graphicx,subcaption,relsize}
\usepackage{amssymb,amsfonts,amsmath,amsthm,amscd,dsfont,mathrsfs,mathtools,multirow,microtype,nicefrac,pifont}
\usepackage{float,psfrag,epsfig,color,url,hyperref}
\usepackage{upgreek}
\usepackage[dvipsnames]{xcolor}
\usepackage{epstopdf,bbm,mathtools,enumitem}
\usepackage[toc,page]{appendix}
\usepackage[mathscr]{euscript}
\usepackage[giveninits=true, maxnames=10, sorting=nyt, style=alphabetic]{biblatex}

\usepackage[top=1in, bottom=1in, left=1in, right=1in]{geometry}

\def\balign#1\ealign{\begin{align}#1\end{align}}
\def\baligns#1\ealigns{\begin{align*}#1\end{align*}}
\def\balignat#1\ealign{\begin{alignat}#1\end{alignat}}
\def\balignats#1\ealigns{\begin{alignat*}#1\end{alignat*}}
\def\bitemize#1\eitemize{\begin{itemize}#1\end{itemize}}
\def\benumerate#1\eenumerate{\begin{enumerate}#1\end{enumerate}}

\newenvironment{talign*}
 {\csname align*\endcsname}
 {\endalign}
\newenvironment{talign}
 {\csname align\endcsname}
 {\endalign}

\def\balignst#1\ealignst{\begin{talign*}#1\end{talign*}}
\def\balignt#1\ealignt{\begin{talign}#1\end{talign}}

\let\originalleft\left
\let\originalright\right
\renewcommand{\left}{\mathopen{}\mathclose\bgroup\originalleft}
\renewcommand{\right}{\aftergroup\egroup\originalright}

\def\tinycitep*#1{{\tiny\citep*{#1}}}
\def\tinycitealt*#1{{\tiny\citealt*{#1}}}
\def\tinycite*#1{{\tiny\cite*{#1}}}
\def\smallcitep*#1{{\scriptsize\citep*{#1}}}
\def\smallcitealt*#1{{\scriptsize\citealt*{#1}}}
\def\smallcite*#1{{\scriptsize\cite*{#1}}}

\def\<{\left\langle} 
\def\>{\right\rangle}

\newcommand{\inner}[1]{{\left\langle #1 \right\rangle}} 

\DeclareSymbolFont{rsfs}{U}{rsfs}{m}{n}
\DeclareSymbolFontAlphabet{\mathscrsfs}{rsfs}
\ifdefined\nonewproofenvironments\else
\ifdefined\ispres\else
\newtheorem{theorem}{Theorem}
\newtheorem{lemma}[theorem]{Lemma}

\renewenvironment{proof}{\noindent\textbf{Proof.}\hspace*{.3em}}{\qed \vspace{.1in}}
\newenvironment{proof-sketch}{\noindent\textbf{Proof Sketch}
  \hspace*{1em}}{\qed\bigskip\\}
\newenvironment{proof-idea}{\noindent\textbf{Proof Idea}
  \hspace*{1em}}{\qed\bigskip\\}
\newenvironment{proof-of-lemma}[1][{}]{\noindent\textbf{Proof of Lemma {#1}}
  \hspace*{1em}}{\qed\\}
  \newenvironment{proof-of-proposition}[1][{}]{\noindent\textbf{Proof of Proposition {#1}}
  \hspace*{1em}}{\qed\\}
\newenvironment{proof-of-theorem}[1][{}]{\noindent\textbf{Proof of Theorem {#1}}
  \hspace*{1em}}{\qed\\}
\newenvironment{proof-attempt}{\noindent\textbf{Proof Attempt}
  \hspace*{1em}}{\qed\bigskip\\}

\newtheorem{remark}{Remark}
\fi

\newtheorem{assumption}{Assumption}
\theoremstyle{definition}
\newtheorem{example}[theorem]{Example}

\fi
\makeatletter
\@addtoreset{equation}{section}
\makeatother

\hypersetup{
  colorlinks,
  linkcolor={red!50!black},
  citecolor={blue!50!black},
  urlcolor={blue!80!black}
}

